\documentclass[11pt,a4paper]{article}

\usepackage[T1]{fontenc}

\usepackage{graphicx}
\usepackage{amsthm,amsmath,amsfonts,amssymb,mathrsfs,mathtools}
\usepackage{calligra}
\usepackage{enumitem}
\usepackage{comment}
\usepackage{pb-diagram}

\usepackage{newtxtext,newtxmath}

\usepackage{tikz}
\usepackage{tikz-cd}
\usetikzlibrary{positioning,intersections,calc,arrows.meta,matrix,arrows}

\tikzcdset{
  arrow style=tikz,
  diagrams={>={Straight Barb[scale=0.8]}}
}

\usepackage{hyperref}

\theoremstyle{definition}
\newtheorem{theo}{Theorem}[section]
\newtheorem{prop}[theo]{Proposition}
\newtheorem{cor}[theo]{Corollary}
\newtheorem{defi}[theo]{Definition}
\newtheorem{lem}[theo]{Lemma}

\newtheorem{rem}[theo]{Remark}

\newtheorem{claim}[theo]{Claim}

\DeclareMathOperator{\Sol}{Sol}
\DeclareMathOperator{\coker}{Coker}
\DeclareMathOperator{\im}{Im}
\DeclareMathOperator{\Ker}{Ker}
\DeclareMathOperator{\Hom}{Hom}
\DeclareMathOperator{\intHom}{\mathscr{H}\text{\kern -3pt {\calligra\large om}}\,}
\DeclareMathOperator{\RintHom}{R\kern -2pt \mathscr{H}\text{\kern -3pt {\calligra\large om}}\,}
\DeclareMathOperator{\ori}{\text{\calligra\large or}}

\DeclareMathOperator{\Rc}{\mathbb{R}\text{-}c}
\DeclareMathOperator{\mS}{SS}
\DeclareMathOperator{\codim}{Codim}
\DeclareMathOperator{\supp}{supp}

\newcommand{\indlim}[1]{\text{``}\varinjlim_{#1}\text{''}}

\tikzset{
  labl/.style={anchor=south, rotate=90, inner sep=.5mm}
}
\tikzset{
  symbol/.style={
    draw=none,
    every to/.append style={
      edge node={node [sloped, allow upside down, auto=false]{$#1$}}
    }
  }
}

\newcommand{\function}[5]{%
  \begin{tikzcd}[
    column sep=2em,
    row sep=1ex,
    ampersand replacement=\&
  ]
  #1\colon \&[-3em]
  #2\vphantom{#3} \arrow[r] \&
  #3\vphantom{#2} \\
  \&
  #4\vphantom{#5} \arrow[u,symbol=\in] \arrow[r,mapsto] \&
  #5\vphantom{#4} \arrow[u,symbol=\in]
  \end{tikzcd}%
}

\newcommand{\idfunction}[5]{%
  \begin{tikzcd}[
    column sep=2em,
    row sep=1ex,
    ampersand replacement=\&
  ]
  #1\vphantom{#2} \arrow["#5",r] \&
  #2\vphantom{#1} \arrow[l] \\
  #3\vphantom{#2} \arrow[u,symbol=\in] \arrow[r] \&
  #4\vphantom{#3} \arrow[u,symbol=\in] \arrow[l]
  \end{tikzcd}%
}
\usepackage{indentfirst}

\title{Strong Regularity and Microsupport Estimates for Multi-Microlocalizations of Subanalytic Sheaves}    
\author{Ryosuke Sakamoto
\footnote{E-mail:sakamoto.ryosuke.h0@elms.hokudai.ac.jp}}    
\date{}    

\begin{document}

\maketitle

\begin{abstract}   
We introduce the notion of strong regularity for subanalytic sheaves and establish estimates for the supports and microsupports of their multi-microlocalizations.
As applications, we study subanalytic sheaves of Whitney and temperate holomorphic solutions of regular $\mathcal{D}$-modules along an involutive subbundle. 
In this setting we prove initial value theorems for multi-microlocal objects with growth conditions and division theorems for temperate and Whitney multi-microfunctions.
As a consequence, we obtain a multi-microlocal version of Bochner's tube theorem for solution sheaves of strongly asymptotically developable functions.
\end{abstract}


\section*{Introduction}

\indent Important analytical objects such as asymptotically developable functions and temperate distributions do not form classical sheaves. Therefore, the conventional microlocal methods developed for sheaves \cite{KS90} cannot be applied
directly to these objects. To handle such objects, several extensions of the classical framework have been introduced, notably ind-sheaves and subanalytic sheaves \cite{KS01,Lu1}. In \cite{HPY}, Honda, Prelli and Yamazaki constructed the theory of multi-microlocalizations of subanalytic sheaves along a certain family of submanifolds. Their construction unifies various types of specialization and microlocalization \cite{ST, D} and provides a functorial framework for multi-microlocal objects with growth conditions, including sheaves associated with Majima's asymptotically developable functions (see \cite{Maj, HP}). In particular, they obtained estimates for the microsupport of multi-microlocalizations in terms of the microsupport of the classical sheaf. However, the existing estimates are limited to classical sheaves and do not apply to subanalytic sheaves.
To overcome this difficulty, we introduce a strengthened notion of regularity for subanalytic sheaves, which we call strong regularity. This notion is designed so as to control the support and microsupport of multi-microlocalizations in the subanalytic setting.

The notion of microsupport for ind-sheaves was introduced in \cite{KS01mic}, and its functorial properties were studied in \cite{M}. In \cite{Lu2}, Prelli obtained support estimates for microlocalizations of subanalytic sheaves by using the microsupport in the sense of Kashiwara and Schapira \cite{KS01mic}. We found that it is necessary to strengthen the microlocal conditions to find the estimate of supports (or microsupport) of multi-microlocalizations for subanalytic sheaves. In this article, we introduce a possible approach to find an estimate of supports and microsupport  of multi-microlocalizations for subanalytic sheaves.

One possible approach is to introduce a stronger notion of regularity for subanalytic sheaves than that of Kashiwara and Schapira \cite{KS01mic}. In \cite{M}, Martins proved that the regularity of ind-sheaves in the sense of \cite{KS01mic} has a natural relation with $\mathcal{D}$-modules regular along an involutive subbundle in the sense of Kashiwara and Oshima \cite{KO}. This class of $\mathcal{D}$-modules is not necessarily holonomic. We show that the strengthened notion introduced in this paper is still compatible with this framework. As an application, we deduce several properties of multi-microlocal analytic objects, including Majima's strongly asymptotically developable solutions of regular $\mathcal{D}$-modules.   
Let us also mention that many authors have studied temperate and Whitney solutions of regular $\mathcal{D}$-modules \cite{KFS, Y} using Andronikof's $THom$ and Colin's $\overset{w}{\otimes}$.

With these estimates in hand, we prove inverse image and multi-microlocalization theorem.  Ultimately, we obtain division theorems with coefficients of temperate and Whitney multi-microfunctions for regular $\mathcal{D}$-modules. As an application, we prove Bochner's tube theorem for solution sheaves of strongly asymptotically developable functions. These are multi-microlocal counterparts of the results obtained in \cite{AT}, \cite{BT}, \cite{U}. Let us mention the related papers \cite{SU}, \cite{T} of this topic.

The contents of this paper are as follows.

In Section \ref{Sec: strong}, we introduce the notion of strong regularity and study its fundamental properties. We show a partial Abelian property of strong regularity. We prove that various solutions to $\mathcal{D}$-modules regular along an involutive subbundle in the sense of \cite{KO}, a class not necessarily holonomic, satisfy strong regularity. We also study functorial properties of strong regularity.

In Section \ref{Sec: estimate}, we obtain estimates of microsupports of multi-microlocalizations. We begin with a support estimate and consider some easy consequences of the results.  We prove the microsupport estimate using the techniques developed in \cite{KS01mic}. As a  simple application we prove the microlocal isomorphism of subanalytic sheaves under a non-characteristic condition.  

In Section \ref{Sec: inverse}, restricting ourselves to the normal crossing case, we prove inverse image and multi-microlocalization theorems. This is a multi-microlocal generalization of the \cite[Theorem 6.7.1]{KS90}.  The multi-conic structure of multi-microlocalizations plays a crucial role. Namely, \cite{HPY}'s multi-normal cones behave differently from the \cite{KS90}'s classical normal cone. First we discuss the inverse and multi-microlocalization theorem in classical sheaf case. Then we apply the results in Section \ref{Sec: estimate} to prove the corresponding theorem for strongly regular subanalytic sheaves. 

In Section \ref{Sec: D}, we give an application to $\mathcal{D}$-modules theory in multi-microlocal analysis. By making use of the triangle obtained in \cite{S}, we first prove the initial value problem for multi-microlocal objects including Majima's strongly asymptotically developable solutions to $\mathcal{D}$-modules regular along an involutive subbundle.
Finally, we obtain division theorems with coefficients of temperate and Whitney multi-microfunctions for regular $\mathcal{D}$-modules. This is regarded as an initial value problem for temperate and Whitney multi-microfunctions for regular $\mathcal{D}$-modules. In particular, by applying the division theorem along with the vanishing theorem of \cite{HPYaxv} we obtain multi-microlocal Bochner's tube theorem. 
\subsection{Notations}
We basically follow the notations of \cite{KS90}, \cite{HPY} and \cite{H}.
Let $X$ be a complex manifold, $X_\mathbb{R}$ the underlying real analytic manifold and $\overline{X}$ the complex conjugate manifold. 
We set $\mathscr{O}^t_X := \RintHom_{\rho_! \mathcal{D}_{\overline{X}}}(\rho_! \mathscr{O}_{\overline{X}}, \mathcal{D}b^t_{X_\mathbb{R}})$, $\mathscr{O}^w_X := \RintHom_{\rho_! \mathcal{D}_{\overline{X}}}(\rho_! \mathscr{O}_{\overline{X}}, \mathcal{C}^{\infty, w}_{X_\mathbb{R}})$. Here $\mathcal{D}b^t_{X_\mathbb{R}}$ is a subanalytic sheaf of temperate distributions on $X_\mathbb{R}$ and $\mathcal{C}^{\infty, w}_{X_\mathbb{R}}$ is a subanalytic sheaf of Whitney $C^\infty$-functions.

Let $\mathcal{C}$ denote a small abelian category.
\cite{KS06} defined the functor $J:{D}^b(\mathrm{Ind}(\mathcal{C})) \to (D^b(\mathcal{C}))^\wedge$ by setting for $F \in {D}^b(\mathrm{Ind}(\mathcal{C}))$ and $G \in D^b(\mathcal{C})$
\[
J(F)(G) = \Hom_{{D}^b(\mathrm{Ind}(\mathcal{C}))}(G, F).
\]
We use the category of compact support constructible sheaves $\mathrm{Mod}^c_{\Rc}(k_X)$ as $\mathcal{C}$ in this paper.
\section{Strongly regular subanalytic sheaves}\label{Sec: strong}
In this section, we develop the notion of strong regularities of subanalytic sheaves.
\begin{defi}[strongly regular subanalytic sheaves]
Let $V$ and $\Omega$ be subsets in $T^\ast X$.
We say $F \in D^b(k_{X_{sa}})$ is strongly regular along $V$ on $\Omega$ at $x \in X$ if there exist $F^\prime \in D^b(k_{X_{sa}})$ isomorphic to $F$ in a neighborhood of $x$, a small and filtrant category $I$ and a functor $I \to D^{[a,b]}_{\Rc}(k_X), i \mapsto F_i$ such that $J(F^\prime) \simeq \indlim{i} J(R\rho_*F_i)$ and $\mS(F_i) \cap \Omega \subset V$. If for any $x\in X$ there exists a neighborhood $U$ of $x$ such that $F \in D^b(k_{X_{sa}})$ is strongly regular along $V$ at $x \in X$ on $\pi^{-1}(U)$ we say $F \in D^b(k_{X_{sa}})$ is strongly regular along $V$.
\end{defi}
\begin{prop}\label{abel}
Let $F \xrightarrow{f} G \xrightarrow{g} H \xrightarrow{+1}$ be a distinguished triangle in $D^b(k_{X_{sa}})$.
Suppose there exists a distinguished triangle of functors from a small filtrant inductive category $I$ to $D^{[a,b]}_{\Rc}(k_X)$, $F \mapsto F_i$, $G \mapsto G_i$ and $H \mapsto H_i$, 
\[\begin{tikzcd}
	{F_i} & {G_i} & {H_i} & {{},}
	\arrow["{f_i}", from=1-1, to=1-2]
	\arrow["{g_i}", from=1-2, to=1-3]
	\arrow["{+1}", from=1-3, to=1-4]
\end{tikzcd}\]
such that $J(F )\simeq \indlim{i} J(F_i)$, $f \simeq \indlim{i} f_i$, $J(G) \simeq \indlim{i} J(G_i)$ and $g \simeq \indlim{i} g_i$. Assume moreover $\mS(F_i) \subset V$ and $\mS(G_i) \subset V$ and the correspondence is functorial with respect to morphisms of triangles, that is for any $i \to j$, $j \to k$ and $j \to k$ the following associated diagram is commutative:
\[
\begin{tikzcd}
F_i \arrow[r,"f_i"] \arrow[d] \arrow[dd,bend right] & 
G_i \arrow[r,"g_i"] \arrow[d] \arrow[dd,bend right] & 
H_i \arrow[r,"+1"] \arrow[d] \arrow[dd,bend right] & {} \\ 
F_j \arrow[r, "f_j"] \arrow[d] & 
G_j \arrow[r, "g_j"] \arrow[d] & 
H_j \arrow[r,"+1"'] \arrow[d] & {} \\ 
F_k \arrow[r, "f_k"] & 
G_k \arrow[r, "f_k"] & 
H_k \arrow[r,"+1"] & {}.
\end{tikzcd}
\]

Then $H$ is strongly regular along $V$.
\end{prop}
\proof
By the triangular inequality, we have $\mS(H_i) \subset V$.
Then it remains to prove $J(H) \simeq \indlim{i\in I} J(H_i)$. By \cite[Lemma 2.3.1]{KSW}, this is equivalent to prove
\[
H^k(H)\simeq  \indlim{i\in I} H^k(H_i),
\]
for any $k$. Consider the following long exact sequences:
\[\begin{tikzcd}
	\cdots & {H^k(F)} & {H^k(G)} & {H^k(H)} & \cdots \\
	\cdots & {\indlim{i} H^k(F_i)} & {\indlim{i} H^k(G_i)} & {\indlim{i} H^k(H_i)} & \cdots.
	\arrow[from=1-1, to=1-2]
	\arrow[from=1-2, to=1-3]
	\arrow[from=1-3, to=1-4]
	\arrow[from=1-4, to=1-5]
	\arrow[from=2-1, to=2-2]
	\arrow["\sim", from=2-2, to=1-2]
	\arrow[from=2-2, to=2-3]
	\arrow["\sim", from=2-3, to=1-3]
	\arrow[from=2-3, to=2-4]
	\arrow[dotted, from=2-4, to=1-4]
	\arrow[from=2-4, to=2-5]
\end{tikzcd}\]
By five lemma for each $k$ the dotted arrow is an isomorphism.
\endproof
By the following theorem, we show the solutions to regular $\mathcal{D}$-modules are naturally strongly regular, while Martins \cite{M} proved the regularity in the sense of \cite{KS01mic} in the case of temperate holomorphic functions. Let $X$ be a complex manifold and let $\mathcal{M}$ be a coherent $\mathcal{D}_X$-module. 
We set for short:
\[
\Sol^t(\mathcal{M}):= \RintHom_{\rho_! \mathcal{D}_X}(\rho_! \mathcal{M}, \mathscr{O}_X^t),
\]
\[
\Sol^w(\mathcal{M}):= \RintHom_{\rho_! \mathcal{D}_X}(\rho_! \mathcal{M}, \mathscr{O}_X^w).
\]
\begin{theo}\label{reg_sol}
{\itshape
If $\mathcal{M}$ is a coherent $\mathcal{D}_X$-module regular along an involutive vector subbundle $V$ of $T^\ast X$ in the sense of Kashiwara-Oshima \cite{KO}, then $\Sol^w(\mathcal{M})$ and $\Sol^t(\mathcal{M})$ are strongly regular along $V$.
}
\end{theo}
\proof
We may assume $X = Z \times Y$, for some complex manifolds $Z$ and $Y$, that $f$ is the projection $X \to Y$ and that $V = X \times_Y T^\ast Y$. By \cite[Lemma 3.6]{KFS}, it is enough to prove the following claim: 
\begin{claim}\label{transport}
There exists a small filtrant inductive system $\{F_i\}_{i \in I} \in \mathrm{C}^{[a,b]}(\mathrm{Mod}^c_{\Rc}(k_Y))$ so that $\Sol^\lambda(\mathcal{D}_{X\to Y})$ is quasi-isomorphic to $Q \indlim{i} F_i$ and
\[
\mS(F_i) \subset V, \text{ for all }i \in I.
\]
\end{claim}
In fact, by \cite[Lemma 3.6]{KFS} $\mathcal{M}$ admits a finite resolution locally on $X$
\[\begin{tikzcd}
	0 & {\mathcal{L}^N} & {\mathcal{L}^{N-1}} & \cdots & {\mathcal{L}^0} & {\mathcal{M}} & 0,
	\arrow[from=1-1, to=1-2]
	\arrow["{\varphi_N}", from=1-2, to=1-3]
	\arrow["{\varphi_{N-1}}", from=1-3, to=1-4]
	\arrow["{\varphi_{1}}", from=1-4, to=1-5]
	\arrow[from=1-5, to=1-6]
	\arrow[from=1-6, to=1-7]
\end{tikzcd}\]
each $\mathcal{D}_X$-module $\mathcal{L}^j$ being isomorphic to a finite direct sum of the $\mathcal{D}_X$-module $\mathcal{D}_{X\to Y}$. We obtain the following short exact sequences:
\[\begin{tikzcd}
	0 & {\mathcal{L}^N} & {\mathcal{L}^{N-1}} & {\im \varphi_{N-1} } & 0, \\
	0 & {\im \varphi_{j-1}} & {\mathcal{L}^{j}} & {\im \varphi_{j}} & 0, \\
	0 & {\im \varphi_0} & {\mathcal{L}^0} & {\mathcal{M}} & 0.
	\arrow[from=1-1, to=1-2]
	\arrow[from=1-2, to=1-3]
	\arrow[from=1-3, to=1-4]
	\arrow["\cdots"{description}, draw=none, from=1-3, to=2-3]
	\arrow[from=1-4, to=1-5]
	\arrow[from=2-1, to=2-2]
	\arrow[from=2-2, to=2-3]
	\arrow[from=2-3, to=2-4]
	\arrow["\cdots"{description}, draw=none, from=2-3, to=3-3]
	\arrow[from=2-4, to=2-5]
	\arrow[from=3-1, to=3-2]
	\arrow[from=3-2, to=3-3]
	\arrow[from=3-3, to=3-4]
	\arrow[from=3-4, to=3-5]
\end{tikzcd}\]
We have the distinguished triangle
\[\begin{tikzcd}
	{\Sol^w(\mathcal{L}^N)} & {\Sol^w(\mathcal{L}^{N-1})} & {\Sol^w(\im \varphi_{N-1})} & {{}}.
	\arrow["\Sol^w(\varphi_{N})", from=1-1, to=1-2]
	\arrow[from=1-2, to=1-3]
	\arrow["{+1}", from=1-3, to=1-4]
\end{tikzcd}\]
There exists a small filtrant inductive system $\{F_i\}_{i \in I} \in \mathrm{C}^{[a,b]}(\mathrm{Mod}^c_{\Rc}(k_Y))$ so that $\Sol^\lambda(\mathcal{L}^N)$ is quasi-isomorphic to $Q \indlim{i} F_i$ and
\[
\mS(F_i) \subset V, \text{ for all }i \in I.
\]
Similarly, there exists a small filtrant inductive system $\{G_i\}_{i \in J} \in \mathrm{C}^{[a,b]}(\mathrm{Mod}^c_{\Rc}(k_Y))$ so that $\Sol^\lambda(\mathcal{L}^{N-1})$ is quasi-isomorphic to $Q \indlim{i} G_i$ and
\[
\mS(G_i) \subset V, \text{ for all }i \in J.
\]
By \cite[Proposition 6.1.13]{KS01} we find a small and filtrant category $K$ and a morphism of functors $\varphi$.
Set $H_i := \coker \varphi_i$ for each $i \in K$ and $\{H_i\}_{i \in K}$ is a small filtrant system. Then using Proposition \ref{abel} we have
$\Sol^w(\im \varphi_{N-1})$ satisfies Claim \ref{transport} replaced with $\Sol^\lambda(\mathcal{D}_{X\to Y})$.
Therefore the result follows by repeatedly applying the above argument.

Let us prove the Claim \ref{transport} for the case of Whitney holomorphic functions. By \cite[Corollary A.4.8]{Lu2}, we have
\[
f^{-1}\mathscr{O}_Y^w \simeq \RintHom_{\rho_! \mathcal{D}_X}(\rho_! \mathcal{D}_{X\to Y}, \mathscr{O}_X^w).
\]
There exists a small filtrant inductive system $\{F_i\}_{i \in I} \in \mathrm{C}^{[a,b]}(\mathrm{Mod}^c_{\Rc}(k_Y))$ so that $F:=\mathscr{O}_Y^w$ is quasi-isomorphic to $Q \indlim{i} F_i$. Then by \cite[Theorem 3.1]{KS01mic} and \cite[Proposition 2.3.2]{KSW} we obtain $J(f^{-1}F) \simeq  \indlim{i} J (f^{-1}F_i)$ and one has
\[
\mS(f^{-1}F_i) \subset X \times_Y T^\ast Y, \text{ for all }i \in I.
\]
It follows that $f^{-1}F$ is strongly regular along $V$.

Next we prove the case of $\mbox{Sol}^t(\mathcal{M})$, the proof is similar. In fact, we may use \cite[Corollary A.3.7]{Lu2} in this case. 
We have
\[
f^{-1}\mathscr{O}_Y^t \simeq \RintHom_{\rho_! \mathcal{D}_X}(\rho_! \mathcal{D}_{X\to Y}, \mathscr{O}_X^t).
\]
And the assertion follows.
\endproof
We study some functorial properties of the regularity.
\begin{prop}\label{sec1: pushforward}
Assume $f:Y\to X$ is a closed embedding and identify $Y$ with a submanifold of $X$. Let $V$ be a subset in $T^\ast Y$.
If $F \in D^b(k_{Y_{sa}})$ is strongly regular along $V$, then $Rf_\ast F$ is strongly regular along $f_\pi {}^tf^{\prime-1}(V)$.
\end{prop}
\proof
If $F$ is strongly regular along $V$ there exist $F^\prime \in D^b(k_{Y_{sa}})$ isomorphic to $F$ in a neighborhood of $y$, a small and filtrant category $I$ and a functor $I \to D^{[a,b]}_{\Rc}(k_Y), i \mapsto F_i$ such that $J(F^\prime) \simeq \indlim{i} J(R\rho_*F_i)$ and $\mS(F_i) \subset V$.
By \cite[Proposition 2.3.2]{KSW}, we obtain
\[
J(Rf_\ast F^\prime) \simeq \indlim{i}J(Rf_\ast F_i).
\]
By \cite[Proposition 5.4.4]{KS90} for each $i$ we have
\[
\mS(Rf_\ast F_i)= f_\pi {}^tf^{\prime-1} (\mS(F_i)).
\]
\endproof
\begin{prop}
Let $K \in D^b_{\Rc}(k_X)$. Let $V$ be a subset in $T^\ast X$. If $F \in D^b(k_{X_{sa}})$ is strongly regular along $V$, then
$K \otimes F$ is strongly regular along $\mS(K) \hat{+} V$.
\end{prop}
\proof
The proof is similar to Proposition \ref{sec1: pushforward}. Apply \cite[Proposition 2.3.2]{KSW} and \cite[Corollary 6.4.5]{KS90}.
\endproof
\begin{prop}
Let $f: Y \to X$ be a morphism. Let $V$ be a subset in $T^\ast X$.
If $F$ is strongly regular along $V$ then $f^{-1}F$ and   $f^{!}F$ are strongly regular along $f^\sharp(V)$.
\end{prop}
\proof
The proof is similar to Proposition \ref{sec1: pushforward}. Apply \cite[Proposition 2.3.2]{KSW} and \cite[Corollary 6.4.4]{KS90}.
\endproof
\section{An estimate of supports and microsupports of multi-microlocalizations for strongly regular subanalytic sheaves}\label{Sec: estimate}
In this section, we study an estimation of supports of multi-microlocalizations. Let $\chi$ be a family of submanifolds $\{M_1, \dots, M_\ell\}$ of a manifold $X$ which satisfies the conditions H1, H2 and H3 of \cite{HPY}.
\subsection{An estimate of supports}\label{sec:supports}

\begin{theo}\label{estimate1}
{\it
Let $V$ be a subset in $T^\ast X$. Assume $F \in D^b(k_{X_{sa}})$ is strongly regular along $V$. Then }
\[
\mu^{sa}_\chi(F)_{C_{\chi^\ast}(V) \cap S^\ast_\chi} \xleftarrow{\sim} \mu_\chi^{sa} (F).
\]
\[
\supp(\rho^{-1}\mu_{\chi}^{sa} F) \subset C_{\chi^\ast}(V) \cap S^\ast_\chi.
\]

\end{theo}
\proof
Set $Z:= C_{\chi^\ast}(V) \cap S^\ast_\chi $.
By the strong regularity, there exist $F^\prime$ isomorphic to $F$ in a neighborhood $U$ of $x$, a small and filtrant category $I$ and a functor $I \to D^{[a,b]}_{\Rc}(k_X), i \mapsto F_i$ such that $J(F^\prime) \simeq \indlim{i} J(R\rho_*F_i)$ and $\mS(F_i) \cap \Omega \subset V$.
By \cite[Theorem 3.3]{KS01mic},
\begin{eqnarray*}
J(\mu_\chi^{sa} (F^\prime)_Z) &\simeq& \indlim{i} J (R\rho_*\mu_\chi (F_i) \otimes k_{Z})\\
&\simeq& \indlim{i} J (R\rho_*\mu_\chi (F_i))\\
&\simeq&J( \mu_\chi^{sa} (F^\prime)).
\end{eqnarray*}
The second isomorphism follows from \cite[Theorem 3.6]{HPY}.
By \cite[Corollary 15.4.4]{KS06}, $J$ is conservative. Thus $\mu_\chi^{sa} (F^\prime)_Z \simeq \mu_\chi^{sa} (F^\prime)$. Therefore for any $x \in X$, there is a neighborhood $U$ so that
\[
\mu^{sa}_\chi(F)_{C_{\chi^\ast}(V) \cap S^\ast_\chi} \simeq \mu_\chi^{sa} (F)
\]
on $\pi^{-1}(U)$.

Let $p\not \in C_{\chi^\ast}(V) \cap S^\ast_\chi$.
By the strong regularity along $V$, there exist a small filtrant system $\{F_i\}$ in $C^{[a,b]}(\text{Mod}^c_{\Rc}(k_X))$ with $\mS(F_i) \subset V$ such that $F$ is quasi-isomorphic to $\varinjlim_i \rho_\ast F_i$ on a neighborhood $W$ of $\pi(p)$. And also we find a multi-conic open neighborhood $U$ along $\chi$ of $p$ such that $C_{\chi^\ast}(V) \cap U =\emptyset$ and $U\subset \pi^{-1}(W)$.
 We have:
\[
H^k\mu^{sa}_\chi(F_W) \simeq \varinjlim_i \rho_\ast H^k  \mu_\chi F_{iW},
\]
for any $k \in \mathbb{Z}$.
Then $\mu^{sa}_\chi(F)|_{U}=0$ since 
$\supp(\mu_\chi F_{i}) \subset C_{\chi^\ast}(\mS(F_i)) \cap S^\ast_\chi$ by \cite[Theorem 3.6]{HPY} and $C_{\chi^\ast}(\mS(F_i)) \subset C_{\chi^\ast}(V)$.
\endproof
\begin{prop}\label{estimate2}
{\it
Let $G_1,\dots,G_\ell \in D^b_{\Rc}(k_X)$, $F\in D^b(k_{X_{sa}})$. Assume $F$ is strongly regular along $V$. 
Set $Z:=C_{\hat{\chi}^\ast}(\mS(G_1)^a \times\dots \times \mS(G_\ell)^a\times i_{\Delta\pi}{}^ti_{\Delta}^{\prime-1}(V))\cap S_{\hat{\chi}}^\ast$, here 
$i_{\Delta\pi}{}^ti_{\Delta}^{\prime-1}(V)=\{(x_1,\dots,x_n; \xi_1, \dots, \xi_n); x_1=\cdots =x_n, \xi_1+ \cdots +\xi_n \in V\}$ in local coordinates. Then we have
\[
\mu hom^{sa}_{\hat{\chi}}(G_1, \dots,G_\ell ; F)_Z \simeq \mu hom^{sa}_{\hat{\chi}}(G_1, \dots,G_\ell ; F).
\]
}
\end{prop}
\proof
By \cite[Theorem 3.3]{KS01mic} and \cite[Proposition 5.4.1, Proposition 5.4.2, Proposition 5.4.4]{KS90}, $\RintHom(p_2^{-1}(G_1\boxtimes \cdots \boxtimes G_\ell), p_{1}^{!}Ri_{\Delta \ast}F))$ is strongly regular along $\mS(G_1)^a \times \cdots \times \mS(G_\ell)^a \times i_{\Delta\pi}{}^ti_{\Delta}^{\prime-1}(V)$. Thus by Theorem \ref{estimate1} the result follows.
\endproof
\begin{lem}\label{Sec2: formula1}
{\it
Let $A$ be a subset in $T^\ast X$. We have:
\[
C_{{\chi}^\ast}(A )\cap S_{{\chi}}^\ast \subset \overline{A}^{{\chi}} \cap S_\chi^\ast.
\]
Here for a subset $D$ in $T^\ast X$ we define $(x^{(0)},0, \dots, 0; 0, \xi^{(1)}, \dots, \xi^{(\ell)}) \in \overline{D}^{{\chi}}$ if and only if there exist sequences $\{(t_{1,n}, \dots, t_{\ell,n})\} \subset (\mathbb{R}^+)^\ell$ and $\{(x^{(0)}_n,\dots, x^{(\ell)}_n;  \xi^{(0)}, \dots,\xi^{(\ell)})\} \subset D$ such that $(t_{\hat{J}_0,n} x^{(0)}_n, t_{\hat{J}^c_{1},k}\xi^{(1)}_k, \dots,t_{\hat{J}^c_{\ell},k}\xi^{(\ell)}_k) \to (x^{(0)}, \xi^{(1)}, \xi^{(2)}, \cdots, \xi^{(\ell)})$ and $t_{j,n} \to +\infty$, $j= 1, \dots, \ell$.
}
\end{lem}
\proof
By \cite[Lemma 1.4]{HPY}, $(x^{(0)},\dots, x^{(\ell)}; \xi^{(0)}, \dots,\xi^{(\ell)})\in C_{\chi^\ast}(A)$ if and only if there exist sequences $$\{(x^{(0)}_k,\dots, x^{(\ell)}_k; \xi^{(0)}_k, \dots,\xi^{(\ell)}_k)\} \subset A$$ and $(t_{1,k}, \dots ,t_{n,k}) \subset (\mathbb{R}^+)^\ell$ such that
\begin{eqnarray}
(t_{\hat{J}_{0},k}x^{(0)}_{k},\dots, t_{\hat{J}_{\ell},k}x^{(\ell)}_{k}; t_{\hat{J}^c_{0},k}\xi^{(0)}_k, \dots,t_{\hat{J}^c_{\ell},k}\xi^{(\ell)}_k) \to (x^{(0)},\dots, x^{(\ell)}; \xi^{(0)}, \dots,\xi^{(\ell)})
\end{eqnarray}
and $t_{1,n},\dots, t_{\ell,n} \to \infty$. This gives the result.
\endproof
\begin{cor}
{\it 
Let $G_1,\dots, G_\ell \in D^b_{\Rc}(k_X)$, $F\in D^b(k_{X_{sa}})$. Assume $F$ is strongly regular along $V$. 
Suppose that $\mS(G_1)\times \dots \times \mS(G_\ell)\cap i_{\Delta\pi}{}^ti_{\Delta}^{\prime-1}(V) \subset T_X^\ast X\times \cdots \times T_X^\ast X$. Then
\[
D^\prime G_1 \otimes \cdots \otimes D^\prime G_\ell \otimes  F \xrightarrow{\sim} \RintHom(G_1\otimes \cdots \otimes G_\ell , F).
\]
}
\end{cor}
\proof
By Lemma \ref{Sec2: formula1}, by the assumption we have:
\begin{eqnarray*}
&{}&C_{\hat{\chi}^\ast}(\mS(G_1)^a \times\dots \times \mS(G_\ell)^a\times i_{\Delta\pi}{}^ti_{\Delta}^{\prime-1}(V))\cap \dot{S}_{\hat{\chi}}^\ast \\
&\subset& \overline{\mS(G_1)\times \dots \times \mS(G_\ell)\cap i_{\Delta\pi}{}^ti_{\Delta}^{\prime-1}(V)}^{{\chi}} \cap \dot{S}_{\hat{\chi}}^\ast\\
&=& \emptyset.
\end{eqnarray*}
Finally use the distinguished triangle in \cite[Corollary 3.12]{S} and Proposition \ref{estimate2}.
\endproof
\begin{cor}
{\it
Assume $F\in D^b(k_{X_{sa}})$ is strongly regular along $V$. Let $U$ and $U^\prime$ be open multi-conic subsets in $S_\chi$ with convex fibers, $U \subset U^\prime$. If $C_{\chi^\ast}(V) \cap U^{\circ a} \subset U^{\prime \circ a}$, then
\[
R\Gamma(U^\prime; \nu_{\chi}(F)) \to R\Gamma(U ;\nu_{\chi}(F))
\]
is an isomorphism.
}
\end{cor}
\proof
Set $\gamma:=U^{\circ a},\gamma^\prime:=U^{\prime \circ a}$. By the following commutative diagram
\begin{eqnarray}
\begin{tikzcd}
	{R\Gamma(U^\prime ; \nu_\chi(F))} & {R\Gamma_{\gamma^\prime}(S^\ast_\chi; \mu_{\chi}(F)\otimes \ori_{M/X}[d])} \\
	{R\Gamma(U ; \nu_\chi(F))} & {R\Gamma_{\gamma}(S^\ast_\chi;\mu_{\chi}(F)\otimes \ori_{M/X}[d])}
	\arrow["\beta",from=1-1, to=2-1]
	\arrow[from=1-2, to=1-1]
	\arrow["\alpha", from=1-2, to=2-2]
	\arrow[from=2-2, to=2-1]
\end{tikzcd}
\end{eqnarray}
with horizontal isomorphisms by \cite[Remark 3.2.7, Proposition 3.2.8]{Lu2} and \cite[Proposition 2.2]{HPY}, where $\alpha$ and $\beta$ are morphisms induced by restrictions. Therefore by Proposition \ref{estimate1}, $\alpha$ is an isomorphism.
\endproof
\subsection{Microsupport estimate of multi-microlocalization}
By making use of the techniques developed in \cite{KS01mic}, we may obtain the microsupport estimate of multi-microlocalization. Recall \cite[Definition 4.4]{KS01mic}: let $\Lambda_i, i \in I$ be a family of closed conic subsets of $T^\ast X$, indexed by the objects of a small and filtrant category $I$.
$p \in T^\ast X$ does not belong to $\lim_i \Lambda_i $ if there exists an open neighborhood $U$ of $p$ and a cofinal subset $J$ of $I$ such that $\Lambda_j \cap U = \emptyset$ for every $j \in J$. 
\begin{theo}\label{sec 2: multi-estimate}
{\itshape
Let $F \in {D}^b(k_{X_{sa}})$ be strongly regular along $V$. We have:
\[
\mS(\mu_\chi^{sa}(F)) \subset \mathrm{C}_{\chi^\ast}(V).
\]
In particular, 
\[
\mS(\rho^{-1}\mu_\chi^{sa}(F)) \subset \mathrm{C}_{\chi^\ast}(V).
\]
}
\end{theo}
\proof
Since $F$ is strongly regular along $V$, for any $x\in X$, we find $F^\prime$ isomorphic to $F$ in a neighborhood $U$ of $x$  so that we have $J(F^\prime)\simeq \indlim{i} J(R\rho_* F_i)$  with $\mS(F_i) \cap \pi^{-1}(U) \subset V$. By \cite[Theorem 3.3]{KS01mic} we have
\[
J(\mu_{\chi}^{sa}(F^\prime)) \simeq  \indlim{i} J(R\rho_* \mu_\chi(F_i)).
\]
By \cite[(4.1)]{KS01mic} and \cite[Theorem 3.6.]{HPY} on $\pi^{-1}U$ we have
\begin{eqnarray*}
\mS(\mu^{sa}_\chi F^\prime) &\subset&\lim_i \mS(\mu_\chi(F_i))\\
&\subset& \lim_i \mathrm{C}_{\chi^\ast}(\mS(F_i)) \subset \mathrm{C}_{\chi^\ast}(V).
\end{eqnarray*}
By \cite[Proposition 3.13 (ii)]{M}, the second assertion follows.
\endproof
\begin{cor}\label{submicroestri}
{\itshape
Let $F \in \mathrm{D}^b(k_{X_{sa}})$ be strongly regular along $V$. We have:
\[
\mS(\mu^{sa}_M(F)) \subset \mathrm{C}_{T^\ast_M X}(V).
\]
}
\end{cor}
\begin{theo}\label{cor6.4.4}
{\itshape
Let $W$ be a subset of $T^\ast Y$. Let $f: Y \to X$ be a morphism of analytic manifolds. Assume $F\in D^b(k_{X_{sa}})$ is strongly regular along $V$ in $T^\ast X$. 
If
$f$ is non-characteristic for $V$ on $W$
then
\[
\rho^{-1} f^{-1}F \otimes \omega_{Y/X} \to \rho^{-1}f^{!} F
\]
is an isomorphism on $W$.
}
\end{theo}
\proof
Identify $Y$ with the graph of $f$ in $X \times Y$ and denote by $q$ the projection $T^\ast_Y(X \times Y) \to Y$. Then by \cite[Theorem 5.1.2]{Lu2} we get:
\begin{equation}
  {}
  \begin{cases}
    f^{-1}F \otimes \omega_{Y/X} \simeq Rq_{!!} \mu_Y(F \boxtimes \omega_Y),\\
    f^! F \simeq Rq_\ast \mu_Y(F \boxtimes \omega_Y).
  \end{cases}
\end{equation}
We have the following distinguished triangle:
\[\begin{tikzcd}
	{\rho^{-1}f^{-1}F \otimes \omega_{Y/X}} & {\rho^{-1}f^{!}F} & {R\dot{q}_\ast \rho^{-1}\mu^{sa}_Y(F)} & {{}.}
	\arrow[from=1-1, to=1-2]
	\arrow[from=1-2, to=1-3]
	\arrow["{+1}", from=1-3, to=1-4]
\end{tikzcd}\]
And by the assumption we have:
\[
\dot{q}_{\pi} {}^t \dot{q}^{\prime-1}(\mathrm{C}_{T^\ast_Y X}(V)) \cap W = \emptyset.
\]
Therefore the result follows from Corollary \ref{submicroestri} and \cite[Proposition 5.5.4]{KS90}. 
\endproof
\begin{cor}
{\itshape
Let $G_1,\dots ,G_\ell \in D^b_{\Rc}(X)$. 
Assume $F\in D^b(k_{X_{sa}})$ is strongly regular along $V$. Then:
\[\mS(\mu hom_{\hat{\chi}}(G_1, \dots, G_\ell; F)) \subset \mathrm{C}_{\hat{\chi}^\ast}(\mS(G_1 \boxtimes \cdots\boxtimes G_\ell)^a \times i_{\Delta\pi}{}^ti_{\Delta}^{\prime-1}(V)).\]
Moreover, 
\[
{\overset{\ell}{\underset{i=1}{\otimes} }D^\prime G_i \otimes \rho^{-1}F} \rightarrow {\rho^{-1}\RintHom(\overset{\ell}{\underset{i=1}{\otimes} }G_i, F)}
\]
is an isomorphism on $T^\ast X \setminus R\dot{\pi}_\pi {}^t \dot{\pi}^{\prime-1} \mathrm{C}_{\hat{\chi}^\ast}(\mS(G_1\boxtimes \cdots \boxtimes G_\ell)^a \times i_{\Delta\pi}{}^ti_{\Delta}^{\prime-1}(V))$.
}
\end{cor}
\proof
Consider the distinguished triangle
\[
\begin{tikzcd}
	{\overset{\ell}{\underset{i=1}{\otimes} }D^\prime G_i \otimes \rho^{-1}F} \arrow[r] & 
	{\rho^{-1}\RintHom(\overset{\ell}{\underset{i=1}{\otimes} }G_i, F)} \arrow[r] & 
	{H} \arrow[r, "+1"] & {,}
\end{tikzcd}
\]
where $H:=R\dot{\pi}_*\mu hom_{\hat{\chi}}(G_1, \dots, G_\ell; F)$.
\endproof

\section{Multi-microlocalizations and inverse images}\label{Sec: inverse}

In this section, our goal is division theorems for regular $\mathcal{D}$-modules. For this purpose we study a theorem of multi-microlocalizations and inverse images which is a generalization of \cite[Theorem 6.7.1]{KS90} in the context of the theory of multi-microlocalizations.
Let $f: Y \to X$ be a morphism of manifolds, let $\chi^M=\{M_1, \dots, M_\ell\}$ and $\chi^N=\{N_1, \dots, N_\ell\}$ be two families of closed submanifolds of $X$ and $Y$ respectively which satisfy H1, H2 and H3 with $f(N_j) \subset M_j$, $j=1, \dots, \ell$.
By \cite[Proposition 2.5]{H} for $F \in D^b(k_{X_{sa}})$ there exists a commutative diagram of canonical morphisms:
\begin{eqnarray}\label{cd}
\begin{tikzcd}
	{R{}^tf^{\prime}_{\chi!}(\omega_{N/M} \otimes f^{-1}_{\chi\pi} \mu^{}_{\chi^M}F)} & {\mu_{\chi^N}(\omega_{Y/X}\otimes f^{-1}F)} \\
	{R{}^tf^{\prime}_{\chi\ast} f^!_{\chi\pi} \mu^{}_{\chi^M}F} & {\mu^{}_{\chi^N}( f^{!}F).}
	\arrow[from=1-1, to=1-2]
	\arrow[from=1-2, to=2-2,"\alpha" ]
	\arrow[from=2-2, to=2-1]
	\arrow[from=1-1, to=2-1, "\beta"]
\end{tikzcd}
\end{eqnarray}
Let us consider $\chi =\{M_1, \dots, M_\ell\}$ is of normal crossing type. $M:= \bigcap_{M_i \in \chi } M_i $ ($N:= \bigcap_{N_i \in \chi } N_i $).
Let $\pi: S^\ast_\chi \to T^\ast_M X$ be an identification in an obvious manner:
\begin{eqnarray}\label{pi}
\function{\pi}{S^\ast_\chi}{T^\ast_M X}{(x, \xi_1), (x, \xi_2), \dots, (x, \xi_\ell)}{(x, \xi_1, \dots, \xi_\ell).}
\end{eqnarray}
We get the following diagram if $\chi^M$ and $\chi^N$ are of normal crossing type:
\[\begin{tikzcd}
	{S^\ast_{\chi^N}} & {N \times_M S_{\chi^M}^\ast}& {S_{\chi^M}^\ast} \\
	{T^\ast_N Y} & {N \times_M T^\ast_M X} & {T^\ast_M X.}
	\arrow["\pi","\sim"',from=1-1, to=2-1]
	\arrow["\pi","\sim"',from=1-3, to=2-3]
	\arrow["\sim"',from=1-2, to=2-2]
	\arrow["f_{\chi \pi}",from=1-2, to=1-3]
	\arrow["{{}^t f^\prime_\chi}"', from=1-2, to=1-1]
	\arrow["{{}^t f^\prime_N}"', from=2-2, to=2-1]
	\arrow["{f_{N\pi}}", from=2-2, to=2-3]
\end{tikzcd}\]
\subsection{Classical sheaf case}
We first study multi-microlocalizations and inverse images in the case of classical sheaves. 
\begin{lem}\label{non-char}
Let $V$ be a closed conic subset of $T^\ast X$.
If
\[V \cap \dot{T}^\ast_M X = \emptyset,\]
then
\[
C_{\chi^{M\ast}}(V)\cap \dot{S}^\ast_{\chi^M} = \emptyset.
\]
\end{lem}
\proof
Assume that there exists a point
\[
(x^{(0)}; \xi^{(1)}, \ldots, \xi^{(\ell)}) \in \dot{S}^*_{\chi^M} \cap C_{\chi^{M *}}(V).
\]

Then there exist sequences
\[
\{(x^{(0)}_n, x^{(1)}_n, \ldots, x^{(\ell)}_n; \xi^{(1)}_n, \ldots, \xi^{(\ell)}_n)\}_{n=1}^{\infty} \subset V,
\]
\[
\{(c^{(1)}_n, \ldots, c^{(\ell)}_n)\}_{n=1}^{\infty} \subset (\mathbb{R}^+)^{\ell},
\]
such that:
\[
\begin{cases}
\displaystyle \lim_{n \to \infty} c^{(j)}_n = \infty, & (j = 1, \ldots, \ell), \\[1em]
\displaystyle \lim_{n \to \infty}
\big(x^{(0)}_n, x^{(1)}_n c^{(\hat{J}_1)}_n, \ldots, x^{(\ell)}_n c^{(\hat{J}_{\ell})}_n;
\xi^{(0)}_n c_n, \xi^{(1)}_n c^{(\hat{J}_1)}_n, \ldots, \xi^{(\ell)}_n c^{(\hat{J}_{\ell})}_n\big)
\\[0.5em]
\quad = (x^{(0)}, 0, \ldots, 0; 0, \xi^{(1)}, \ldots, \xi^{(\ell)}),
\end{cases}
\]
where
\[
c_n := \prod_{j=1}^{\ell} c^{(j)}_n, \quad
\hat{J}_j := \{1, \ldots, \ell\} \setminus J_j, \quad
c^{(J)}_n := \prod_{j \in J} c^{(j)}_n \quad \text{for any } J \subset \{1, \ldots, \ell\}.
\]
In particular, we have
\[
\lim_{n \to \infty} (x^{(1)}_n, \ldots, x^{(\ell)}_n, \xi^{(0)}_n c_n) = (0, \ldots, 0, 0).
\]

Since $(\xi^{(1)}, \ldots, \xi^{(\ell)}) \neq 0$, we may assume that
\[
t_n := |(\xi^{(0)}_n, \xi^{(1)}_n, \ldots, \xi^{(\ell)}_n)| > 0.
\]

We consider the sequence
\[
\{(x^{(0)}_n, x^{(1)}_n, \ldots, x^{(\ell)}_n; t_n^{-1}(\xi^{(0)}_n, \xi^{(1)}_n, \ldots, \xi^{(\ell)}_n))\}_{n=1}^{\infty}
\subset V.
\]
By extracting a subsequence, we may assume that there exists $\zeta_0 \neq 0$ such that
\[
\lim_{n \to \infty} t_n^{-1}(\xi^{(0)}_n, \xi^{(1)}_n, \ldots, \xi^{(\ell)}_n) = \zeta_0.
\]
Choose $1 \leq k \leq \ell$ with 
\[
\lim_{n \to \infty} \xi^{(k)}_n c^{(\hat{J}_k)}_n = \xi^{(k)} \neq 0.
\]
Since $\lim_{n \to \infty} c^{(j)}_n = \infty$, we have
\[
\lim_{n \to \infty} \frac{\xi^{(0)}_n}{\xi^{(k)}_n}
= \lim_{n \to \infty} \frac{\xi^{(0)}_n c_n}{\xi^{(k)}_n c^{(\hat{J}_k)}_n c^{(J_k)}_n} = 0.
\]

Therefore, we see that $\zeta_0 = (0, \zeta_0^{(1)}, \zeta_0^{(2)}, \cdots, \zeta_0^{(\ell)}) \neq 0$. Hence
\[
\lim_{n \to \infty} (x^{(0)}_n, x^{(1)}_n, \ldots, x^{(\ell)}_n;
t_n^{-1}(\xi^{(0)}_n, \xi^{(1)}_n, \ldots, \xi^{(\ell)}_n))
= (x^{(0)}, 0, \ldots, 0; \zeta_0) \in \dot{T}^*_M X \cap V,
\]
which contradicts the non-characteristic condition.
\endproof
\begin{theo}\label{multi-inverse}
{\itshape
Let any family of submanifolds be of normal crossing type.
Let $F \in D^b(X)$ and $V$ be an open subset of $T^\ast_N Y$. Assume:
\begin{enumerate}[label=(\roman*)]
\item $f$ is non-characteristic for $F$ on $V$,
\item $S^\ast_{\chi^N} \cap C_{\chi^{N\ast}}(f^\sharp_\infty (\mS(F))) \cap \pi^{-1}(V) = \emptyset$,
\item $f_{N\pi}: {}^t f_N^{\prime-1}(V) \to T^\ast_M X$ is non-characteristic for $C_{T^\ast_MX}(\mS(F))$,
\item $f_{\chi \pi}: {}^t f^{\prime-1}_\chi(\pi^{-1}(V)) \to S_{\chi^{M}}^\ast$ is non-characteristic for $C_{\chi^{M\ast}}(\mS(F))$,
\item ${}^t f^{\prime-1}(V)\cap f^{-1}_\pi (\mS(F)) \subset Y \times_X T^\ast_M X$,
\item ${}^t  f'_{\chi}$ is proper on ${}^t f^{\prime-1}_\chi(\pi^{-1}(V))\cap f^{-1}_{\chi \pi}(C_{\chi^{M\ast}}(\mS(F))\cap S^\ast_{\chi^M})$.
\end{enumerate}
Then
\[R{}^t \kern -2pt f_{\chi!}^{\prime}(\omega_{N/M}\otimes f^{-1}_{\chi \pi}\mu_{\chi^M}(F))|_{\pi^{-1}(V)} \to \mu_{\chi^N}(\omega_{Y/X} \otimes f^{-1}F)|_{\pi^{-1}(V)}\]
and
\[\mu_{\chi^N}(f^! F)|_{\pi^{-1}(V)} \to R{}^tf^{\prime}_N f^!_{N\pi} \mu_{\chi^M}(F)|_{\pi^{-1}(V)}\]
are isomorphisms.
}
\end{theo}
\proof
It is enough to prove in the case that $f$ is a closed embedding.
By \cite[Theorem 3.6]{HPY}, \cite[Corollary 6.4.4]{KS90} and the assumption (iv), $\omega_{N/M} \otimes f^{-1}_{\chi\pi} \mu^{}_{\chi_M}F \to f^!_{\chi\pi} \mu^{}_{\chi^M}F$ is an isomorphism on ${}^t f^{\prime-1}_\chi(\pi^{-1}(V)) $.
By the assumption (vi), ${}^t \kern -2pt f^\prime_{\chi}$ is proper on ${}^t f^{\prime-1}_\chi(\pi^{-1}(V))\cap \supp(f^{-1}_{\chi\pi} \mu^{}_{\chi_M}F )$.
Thus the vertical arrow $\beta$ in \eqref{cd} is an isomorphism on $\pi^{-1}(V)$. 
By the assumption (ii), the vertical arrow $\alpha$ in \eqref{cd} is an isomorphism on $\pi^{-1}(V)$. In fact there is a following distinguished triangle:
\[\begin{tikzcd}
	{\omega_{Y/X} \otimes f^{-1}F} & {f^{!}F} & {G} & {{},}
	\arrow[from=1-1, to=1-2]
	\arrow[from=1-2, to=1-3]
	\arrow["{+1}", from=1-3, to=1-4]
\end{tikzcd}\]
where $G:= R\dot{q}_\ast \mu_Y F$ and $q$ is the projection $T^\ast_Y(X \times Y) \to Y$. Applying $\mu_{\chi^N}$ we obtain the following distinguished triangle:

\[\begin{tikzcd}
	{\mu_{\chi^N}(\omega_{Y/X} \otimes f^{-1}F)} & {\mu_{\chi^N}(f^{!}F)} & {\mu_{\chi^N} (G)} & {{}.}
	\arrow[from=1-1, to=1-2]
	\arrow[from=1-2, to=1-3]
	\arrow["{+1}", from=1-3, to=1-4]
\end{tikzcd}\]
Then by \cite[Theorem 3.6]{HPY}, \cite[Proposition 5.5.4]{KS90} and the assumption (ii), $\supp(\mu_{\chi^N} (G)) \cap \pi^{-1}(V) \subset S^\ast_{\chi^N} \cap C_{\chi^{N\ast}}(\dot{q}_\pi {}^t \dot{q}^{\prime-1} C_{T^\ast_Y X}(\mS(F))) \cap \pi^{-1}(V) = S^\ast_{\chi^N} \cap C_{\chi^{N\ast}}(f^\sharp_\infty(\mS(F))) \cap \pi^{-1}(V)=\emptyset$.

Therefore, it suffices to prove one of the horizontal arrows in (\ref{cd}). Consider following chain of isomorphisms:
\begin{eqnarray*}
T_\chi f^{-1}\nu_{\chi^M} (F) &\simeq& T_\chi f^{-1} s^{-1}_X R\Gamma_{\Omega_X} p^{-1}F\\
&\simeq& s^{-1}_Y \tilde{f}^{\prime-1} R\Gamma_{\Omega_X} p^{-1}F\\
&\simeq& s^{-1}_Y \tilde{f}^{\prime-1} Rj_{X\ast}\tilde{p}_X^{-1} F
\end{eqnarray*}
and
\begin{eqnarray*}
\nu_{\chi^N}(f^{-1}F)&\simeq& s^{-1}_Y Rj_{Y\ast} \tilde{f}^{-1}\tilde{p}^{-1}_X F\\
&\simeq& s^{-1}_Y(\tilde{f}^{\prime !} Rj_{X\ast} \tilde{p}^{-1}_X F \otimes \omega_{\tilde{Y}/\tilde{X}}^{\otimes -1}).
\end{eqnarray*}
Consider a distinguished triangle:
\[\begin{tikzcd}
	{\omega_{\tilde{Y}/\tilde{X}} \otimes \tilde{f}^{\prime-1}Rj_{X\ast}\tilde{p}^{-1}_XF} & {\tilde{f}^{\prime!}Rj_{X\ast}\tilde{p}^{-1}_XF} & H & {{}.}
	\arrow[from=1-1, to=1-2]
	\arrow[from=1-2, to=1-3]
	\arrow["{+1}", from=1-3, to=1-4]
\end{tikzcd}\]
Then $H$ is supported in $S_{\chi^N}$. 
And it is enough to prove $\pi^{-1}_{}(V) \cap \mS((s_Y^{-1}H)^\wedge) = \emptyset$. 
Recall the identification  \cite[Theorem 3.3]{HPY}.
 Thus by the identification (i.e., We perform the proof in the space $T^\ast S_{\chi^N} $.),  it is enough to prove $\pi^{-1}_{}(V) \cap \mS(H)= \emptyset$. By \cite[Corollary 6.4.4]{KS90} it suffices to show:
\begin{claim}
for any $q_0 \in \pi^{-1}(V)$, there exists $q \in S_\chi \times_{\tilde{Y}} T^\ast \tilde{Y}$ such that $q_0 = {}^t s_Y^\prime(q) \in \pi^{-1}(V)$, and $\tilde{f}^\prime$ is non-characteristic for $Rj_{X\ast} \tilde{p}^{-1}_XF$ at $q$.
\end{claim}
We choose local coordinates systems $(x)=(x^{(0)}, x^{(1)}, \dots ,x^{(\ell)})$ on $X$, $(y)=(y^{(0)}, y^{(1)}, \dots ,y^{(\ell)})$ on $Y$ such that $M_k =\{x^i=0\, (i \in I_k)\}$, $N_k =\{y^i=0\,(i \in I_k)\}$. We denote by 
\[(x^{(0)}, x^{(1)}, \dots ,x^{(\ell)}, t_1,\dots ,t_\ell; \xi^{(0)}, \xi^{(1)}, \dots ,\xi^{(\ell)}, \tau_1,\dots, \tau_\ell),
\] 
the coordinates on $\tilde{X}$, by 
\[
(y^{(0)}, y^{(1)}, \dots ,y^{(\ell)}, t_1,\dots ,t_\ell; \eta^{(0)}, \eta^{(1)}, \dots, \eta^{(\ell)}, \tau_1,\dots, \tau_\ell),
\]
the coordinates on $\tilde{Y}$. We write:
\[f(y)=(g_0(y), g_1(y), \dots ,g_\ell(y)),\]
\[\tilde{f}(y, t_1,\dots, t_\ell)=(\tilde{g}_0(y, t_1, \dots, t_\ell), \tilde{g_1}(y,t_1, \dots, t_\ell), \tilde{g_\ell}(y ,t_1,\dots ,t_\ell), t_1,\dots,  t_\ell).\]
Hence:
\[
t_{J_k}\tilde{g}_k(y^{(0)}, y^{(1)}, \dots ,y^{(\ell)},t_1,\dots, t_\ell)= g_k(y^{(0)}, t_{J_1 }y^{(1)}, \dots,t_{J_\ell}y^{(\ell)}),
\]
\[
\tilde{g}_0(y^{(0)}, y^{(1)}, \dots ,y^{(\ell)},t_1,\dots, t_\ell)= g_0(y^{(0)}, t_{J_1 }y^{(1)}, \dots,t_{J_\ell}y^{(\ell)}).
\]
Note that:
\begin{eqnarray}\label{sec3: calc1}
&{}&{}^tf^\prime (y^{(0)}, ty^{(1)}, \dots ,ty^{(\ell)})\cdot(t \xi^{(0)},\xi^{(1)}, \dots,\xi^{(\ell)})\\
&=& \begin{pmatrix}
\tilde{g}_{0, y^{(0)}}(y,t)\cdot t\xi^{(0)} + \tilde{g}_{1, y^{(0)}}(y,t)\cdot t\xi^{(1)} + \cdots+\tilde{g}_{\ell, y^{(0)}}(y,t)\cdot t\xi^{(\ell)} \\
\tilde{g}_{0, y^{(1)}}(y,t)\cdot \xi^{(0)} + \tilde{g}_{1, y^{(1)}}(y,t)\cdot \xi^{(1)} + \cdots+\tilde{g}_{\ell, y^{(1)}}(y,t)\cdot \xi^{(\ell)} \\
\cdots\\
\tilde{g}_{0, y^{(\ell)}}(y,t)\cdot \xi^{(0)} + \tilde{g}_{1, y^{(\ell)}}(y,t)\cdot \xi^{(1)} + \cdots+\tilde{g}_{\ell, y^{(\ell)}}(y,t)\cdot \xi^{(\ell)}
\end{pmatrix}\nonumber.
\end{eqnarray}
We set $q = (y_0, 0,0,0,0; 0, \eta_1,\dots ,\eta_\ell, \tau_1, \dots,\tau_\ell)$, $x_0 = g_0(y_0, 0,0)$, $q_0 = (y_0; \eta_1, \dots,\eta_\ell) \in \pi^{-1}(V)$.
Let us assume $\tilde{f}^\prime$ is characteristic for $Rj_{X\ast}\tilde{p}^{-1}_X F$ at $q$.
Applying \cite[Proposition 6.2.4 (iii) (b)]{KS90} ($A = T^\ast_{\tilde{Y}}\tilde{Y}$, $B= \mS(Rj_{X\ast} \tilde{p}^{-1}_X F)$) we find sequences:
\[
\{(y^{(0)}_n, y^{(1)}_n, \dots,y^{(\ell)}_n, \hat{t}^{1}_n, \dots, \hat{t}^{\ell}_n)\} \subset \tilde{Y}
\]
and
\[
\{(x^{(0)}_n, x^{(1)}_n, \dots, x^{(\ell)}_n, t^{1}_n,\dots ,t^{\ell}_n; \xi^{(0)}_n, \xi^{(1)}_n, \dots,\xi^{(\ell)}_n, \tau_n^{1},\dots ,\tau_n^{\ell})\} \subset \mS(Rj_{X\ast} \tilde{p}^{-1}_X F),
\]
such that:
\begin{equation}
  {}
  \begin{dcases}
    (y^{(0)}_n, y^{(1)}_n, \dots, y^{(\ell)}_n, \hat{t}^{1}_n, \dots, \hat{t}^\ell_n) \xrightarrow{n} (y_0, 0,0),\\
(x^{(0)}_n, x^{(1)}_n, \dots,x^{(\ell)}_n, t^1_n,\dots ,t^\ell_n)\xrightarrow{n} (x_0, 0,0),
  \end{dcases}
\end{equation}
\begin{equation}\label{6.7.3}
  {}
  \begin{dcases}
    \tau^1_n +\sum_{k=0}^\ell \tilde{g}_{k,t_1}(y^{(0)}_n, y^{(1)}_n, \dots, y^{(\ell)}_n, \hat{t}^{1}_n, \dots, \hat{t}^\ell_n)\cdot \xi_n^{(k)}  \xrightarrow{n} \tau_1,\\
    \cdots\\
    \tau^\ell_n +\sum_{k=0}^\ell \tilde{g}_{k,t_\ell}(y^{(0)}_n, y^{(1)}_n, \dots, y^{(\ell)}_n,\hat{t}^{1}_n, \dots, \hat{t}^\ell_n)\cdot  \xi_n^{(k)}  \xrightarrow{n} \tau_\ell,
  \end{dcases}
\end{equation}
\begin{equation}\label{condition3}
  {}
  \begin{dcases}
\sum_{k=0}^\ell \tilde{g}_{k, y^0}(y^{(0)}_n, y^{(1)}_n, \dots, y^{(\ell)}_n, \hat{t}^{1}_n, \dots, \hat{t}^\ell_n)\cdot \xi_n^{(k)}  \xrightarrow{n} 0,\\
\sum_{k=0}^\ell \tilde{g}_{k, y^1}(y^{(0)}_n, y^{(1)}_n, \dots, y^{(\ell)}_n, \hat{t}^{1}_n, \dots, \hat{t}^\ell_n)\cdot \xi_n^{(k)}  \xrightarrow{n}\eta_1,\\
\cdots\\
\sum_{k=0}^\ell \tilde{g}_{k, y^\ell}(y^{(0)}_n, y^{(1)}_n, \dots, y^{(\ell)}_n, \hat{t}^{1}_n, \dots, \hat{t}^\ell_n)\cdot \xi_n^{(k)}  \xrightarrow{n}\eta_\ell,
  \end{dcases}
\end{equation}

\begin{eqnarray}\label{6.7.6}
|\xi^{(0)}_n|+\cdots +|\xi^{(\ell)}_n|+|\tau_n^1|+\cdots +|\tau_n^\ell| \xrightarrow{n} \infty,
\end{eqnarray}

\begin{eqnarray}\label{6.7.7}
|(x^{(0)}_n, x^{(1)}_n, \dots,x^{(\ell)}_n, t^1_n,\dots ,t^\ell_n)- \tilde{f}(y^{(0)}_n, y^{(1)}_n, \dots, y^{(\ell)}_n, \hat{t}^{1}_n, \dots, \hat{t}^\ell_n)|\cdot(\sum_{k=0}^\ell|\xi^{(k)}_n|+\sum_{k=1}^\ell|\tau_n^k|) \xrightarrow{n} 0.
\end{eqnarray}
By \eqref{6.7.3}, \eqref{6.7.6} we obtain:
\begin{eqnarray}\label{infty}
|\xi^{(0)}_n| + |\xi^{(1)}_n| + \cdots +|\xi^{(\ell)}_n| \xrightarrow{n} \infty.
\end{eqnarray}
In fact, assume by contradiction that $|\xi^{(0)}_n| + |\xi^{(1)}_n| + \cdots +|\xi^{(\ell)}_n|$ is bounded. By \eqref{6.7.6}, $|\tau_n^1|+\cdots +|\tau_n^\ell| \to +\infty$. We may assume $|\tau_n^k|  \to +\infty$ for some $1 \leq k\leq \ell$. Since $|\xi^{(0)}_n| + |\xi^{(1)}_n| + \cdots +|\xi^{(\ell)}_n|$ is bounded,  $|\tau_n^k|  \to +\infty$ contradicts to \eqref{6.7.3}.

Moreover \eqref{6.7.7} implies in particular:
\begin{eqnarray}\label{6.7.9}
|t_n^j-\hat{t}^j_n|\cdot(|\xi^{(0)}_n|+|\xi^{(1)}_n|+ \cdots +|\xi^{(\ell)}_n|+|\tau_n^1|+ \cdots +|\tau_n^\ell|)\xrightarrow{n} 0.
\end{eqnarray}
Then by Lemma \ref{lem1} below, we may assume $t^j_n >0$ and 
\[
(x^{(0)}_n, -t_n x^{(1)}_n, \dots, -t_n x^{(\ell)}_n ;t_n\xi_n^{(0)}, \xi^{(1)}_n, \dots,\xi^{(\ell)}_n ) \in \mS(F)
\]
(keep in mind of the identification \cite[Theorem 3.3]{HPY}), where $t_n=\prod_{k=1}^\ell t_{k,n}  $. By \eqref{sec3: calc1}, \eqref{condition3} and \eqref{6.7.9} we have:

\begin{eqnarray}\label{6.7.10}
{}^t f^\prime(y^{(0)}_n, -t_n y^{(1)}_n,\dots ,-t_n y^{(\ell)}_n)\cdot (t_n\xi^{(0)}_n, \xi^{(1)}_n,\dots ,\xi^{(\ell)}_n)\xrightarrow{n} (0, \eta_1,\dots ,\eta_\ell).
\end{eqnarray}
Since $f$ is non-characteristic for $F$ at $\pi_{}(q_0)$ (hypothesis (i)), this implies:
\begin{eqnarray}\label{bdd}
\{
|t_n \xi^{(0)}_n|+|\xi^{(1)}_n|+\cdots+|\xi^{(\ell)}_n|
\}
\text{ is bounded.}
\end{eqnarray}
Thus by \eqref{infty} and \eqref{bdd}, $|\xi^{(0)}_n |\xrightarrow{n}  \infty$.
Since the sequence $\{(t_n \xi^{(0)}_n, \xi^{(1)}_n,\dots ,\xi^{(\ell)}_n)\}$ is bounded, we may assume after extracting a subsequence so that $\xi^{(1)}_n \xrightarrow{n} \xi^{(1)}_1$, \dots , $\xi^{(\ell)}_n \xrightarrow{n} \xi^{(\ell)}_1$, $t_n \xi^{(0)}_n\xrightarrow{n} \xi^{(0)}_1$. Then by \eqref{6.7.10}
\[
(y^{(0)},0, 0; \xi^{(0)}_1, \xi^{(1)}_1,\dots,\xi^{(\ell)}_1) \in {}^t f^{\prime-1}(V)\cap f^{-1}_\pi (\mS(F)).
\]
By the hypothesis (v), we get $\xi^{(0)}_1=0$. Thus $(t_n\xi^{(0)}_n, \xi^{(1)}_n, \dots ,\xi_n^{(\ell)})\xrightarrow{n} (0, \xi^{(1)}_1,\dots, \xi^{(\ell)}_1)$. We may assume that $\xi^{(0)}_n/|\xi^{(0)}_n|$ has a limit $\xi^{(0)}_2$. We obtain a sequence \[\{(x^{(0)}_n, -t_n x^{(1)}_n, \dots,-t_n x^{(\ell)}_n ;t_n\xi_n^{(0)}, \xi^{(1)}_n, \dots,\xi^{(\ell)}_n )\}\] in $\mS(F)$ such that:
\begin{eqnarray*}
(x^{(0)}_n, -t_n x^{(1)}_n, \dots,-t_n x^{(\ell)}_n) \xrightarrow{n} (x_0, 0,\dots ,0),\\
(t_n\xi_n^{(0)}, \xi^{(1)}_n, \dots,\xi^{(\ell)}_n )\xrightarrow{n} (0, \xi^{(1)}_1,\dots,\xi^{(\ell)}_1),\\
\frac{1}{t_n |\xi_{n}^{(0)}|}(t_n\xi^{(0)}_n, t_n x^{(1)}_n, \dots,t_n x^{(\ell)}_n)\xrightarrow{n} (\xi^{(0)}_2,0,\dots, 0).
\end{eqnarray*}

Then $(x_0, 0,\dots, 0; \xi^{(0)}_2, \xi^{(1)}_1,\dots,\xi^{(\ell)}_1) \in C_{T^\ast_MX}(\mS(F))$, and hypothesis (iii) implies $g_{0, y^0}(y^0,0,0)\cdot\xi^{(0)}_2=\tilde{g}_{y^0}(y^0,0,0)\cdot\xi_2^{(0)}\not=0$. By \eqref{condition3} we have (recall that $|\xi^0_n| \xrightarrow{n} \infty$):
\[
\sum_{k=0}^\ell \tilde{g}_{k, y^0}(y^{(0)}_n, y^{(1)}_n, \dots, y^{(\ell)}_n, t_n^1,\dots, t_n^\ell)\cdot \frac{\xi^{(k)}_n}{|\xi^{(0)}_n|} \xrightarrow{n} 0.
\]
We get a contradiction, since $|\xi^{(1)}_n|$, $|\xi^{(2)}_n|$ , \dots and $|\xi^{(\ell)}_n|$ are bounded.
\endproof
\begin{lem}\label{lem1}
Suppose there exists a sequence
\[
(x^{(0)}_n, x^{(1)}_n, \dots, x^{(\ell)}_n, t^{1}_n,\dots ,t^{\ell}_n; \xi^{(0)}_n, \xi^{(1)}_n, \dots,\xi^{(\ell)}_n, \tau_n^{1},\dots ,\tau_n^{\ell})  
\]
in $\mS(Rj_{X\ast} \tilde{p}_{X}^{-1} F)$ which satisfies the conditions \eqref{6.7.3}--\eqref{6.7.7}. 
Then we find sequences
\[
\{(x^{(0)}_n, x^{(1)}_n, \dots, x^{(\ell)}_n; \xi^{(0)}_n, \xi^{(1)}_n, \dots,\xi^{(\ell)}_n)\} \subset \mS(F)
\]
and $\{(t^{1}_n,\dots ,t^{\ell}_n)\} \subset (\mathbb{R}^+)^\ell$ so that they satisfy the conditions \eqref{6.7.3}--\eqref{6.7.7}. Also since $\mS(F)$ is conic, we have
\[
(x^{(0)}_n,  x^{(1)}_n, \dots, x^{(\ell)}_n ;t_n\xi_n^{(0)}, t_n \xi^{(1)}_n, \dots, t_n \xi^{(\ell)}_n ) \in \mS(F),
\]
where $t_n=\prod_{k=1}^\ell t_{k,n} $.
\end{lem}
\proof
By \cite[Theorem 6.3.1]{KS90} we have
\[
\mS(Rj_{\Omega_X \ast} \tilde{p}_X^{-1} F) \subseteq \mS(\tilde{p}^{-1}_X F) \hat{+} N^\ast(\Omega_X).
\]
By \cite[Proposition 5.4.5]{KS90} we have
\begin{align*}
\mS(\tilde{p}^{-1}_X F) = \{(x_0, \frac{x^{(0)}}{t_{\hat{J}_1}}, \dots,\frac{x^{(\ell)}}{t_{\hat{J}_{\ell}}}; t_1, t_2,\dots, t_\ell; \xi_0, t_{\hat{J}_1}\xi^{(1)},\dots ,t_{\hat{J}_\ell}\xi^{(\ell)};\tau_1, \tau_2, \dots, \tau_\ell); \\
t_j>0, (x^{(0)}, x^{(1)}, \dots, x^{(\ell)}; \xi^{(0)}, \xi^{(1)}, \dots, \xi^{(\ell)})\in \mS(F)\}.
\end{align*}
Here we do not explicitly calculate $\tau'$s because we do not use them.
By \cite[Remark 6.2.8]{KS90} and the fact that $N^\ast(\Omega_X) \subset \{(x; t; 0; \tau)\}$, for each $n \in \mathbb{N}$, we get sequences
\[
\{(x^{(0)}_{n,m}, x^{(1)}_{n,m}, \dots, x^{(\ell)}_{n,m}; \xi^{(0)}_{n,m}, \xi^{(1)}_{n,m}, \dots,\xi^{(\ell)}_{n,m})\}_{m=1}^\infty \subset \mS(F)
\]
\[
\{(t_{1,n,m}, t_{2,n,m}, \dots, t_{\ell,n,m})\} \subset (\mathbb{R}^+)^\ell
\]
such that 
\[
(x^{(0)}_{n,m}, \frac{x^{(1)}_{n,m}}{t_{\hat{J}_1,n, m}}, \dots,\frac{x^{(\ell)}_{n,m}}{t_{\hat{J}_\ell,n, m}}; t_{1,n,m},\dots,t_{\ell,n,m}) \xrightarrow{m} (x_{n}^{(0)}, x_{n}^{(1)},\dots ,x_{n}^{(\ell)}, t_{1,n}, t_{2,n}, \dots, t_{\ell, n}), 
\]
\[
(\xi^{(0)}_{n,m}, t_{\hat{J}_1,n,m}\xi^{(1)}_{n,m}, \dots, t_{\hat{J}_\ell,n, m}\xi^{(\ell)}_{n,m}) \xrightarrow{m} ( \xi_{n}^{(0)}, \xi_{n}^{(1)}, \dots,\xi_{n}^{(\ell)}).
\]
Then extracting a subsequence we find sequences
\[
\{(x^{(0)}_{n}, x^{(1)}_{n}, \dots,x^{(\ell)}_{n}; \xi^{(0)}_{n}, \xi^{(1)}_{n}, \dots, \xi^{(\ell)}_{n})\}_{n=1}^\infty \subset \mS(F),
\]
and $\{(t_{1,n}, t_{2,n}, \dots, t_{\ell, n})\} \subset (\mathbb{R}^+)^\ell$ so that they satisfy desired conditions.
Since $\mS(F)$ is conic, we have
\[
\{(x^{(0)}_{n},  x^{(1)}_{n}, \dots,  x^{(\ell)}_{n}; t_n \xi^{(0)}_{n}, t_n \xi^{(1)}_{n}, \dots, t_n\xi^{(\ell)}_{n})\}  \subset \mS(F).
\]
\endproof
\begin{rem}
By \cite[Remark 6.2.6]{KS90}, if we replace the condition (i) of Theorem \ref{multi-inverse} with the following condition (1) then we do not need (ii).
\begin{enumerate}
\item[(1)] $f$ is non-characteristic for $F$.
\end{enumerate}
\end{rem}
\subsection{Strongly regular subanalytic sheaf case}
We study multi-microlocalizations and inverse images in the case of strongly regular subanalytic sheaves. 
\begin{theo}\label{multi-inverse-sub}
{\itshape
Let any family of submanifolds be of normal crossing type.
Let $F \in D^b(k_{X_{sa}})$ be a strongly regular along a closed conic subset $W$ and $V$ be an open subset of $T^\ast_N Y$. Assume:
\begin{enumerate}[label=(\roman*)]
\item $f$ is non-characteristic for $W$,
\item $f_{N\pi}: {}^t f_N^{\prime-1}(V) \to T^\ast_M X$ is non-characteristic for $C_{T^\ast_MX}(W)$,
\item $f_{\chi \pi}: {}^t f^{\prime-1}_\chi(\pi^{-1}(V)) \to S_{\chi^{M}}^\ast$ is non-characteristic for $C_{\chi^{M\ast}}(W)$,
\item ${}^t f^{\prime-1}(V)\cap f^{-1}_\pi (W) \subset Y \times_X T^\ast_M X$,
\item ${}^t  f'_{\chi}$ is proper on ${}^t f^{\prime-1}_\chi(\pi^{-1}(V))\cap f^{-1}_{\chi \pi}(C_{\chi^{M\ast}}(W)\cap S^\ast_{\chi^M})$.
\end{enumerate}
Then
\[R{}^t \kern -2pt f_{\chi!}^{\prime}(\omega_{N/M}\otimes f^{-1}_{\chi \pi}\mu_{\chi^M}(F))|_{\pi^{-1}(V)} \to \mu_{\chi^N}(\omega_{Y/X} \otimes f^{-1}F)|_{\pi^{-1}(V)}\]
and
\[\mu_{\chi^N}(f^! F)|_{\pi^{-1}(V)} \to R{}^tf^{\prime}_\chi f^!_{\chi\pi} \mu_{\chi^M}(F)|_{\pi^{-1}(V)}\]
are isomorphisms.
}
\end{theo}
\proof
We may assume $f$ is a closed embedding. 
By Theorem \ref{sec 2: multi-estimate}, \cite[Corollary 6.4.4]{KS90} and the assumption (iii), $\omega_{N/M} \otimes f^{-1}_{\chi\pi} \mu^{}_{\chi^M}F \to f^{!}_{\chi\pi} \mu^{}_{\chi^M}F$ is an isomorphism on ${}^t f^{\prime-1}_\chi (\pi^{-1}(V))$. 

By the assumption (i), the vertical arrow $\alpha$ in \eqref{cd} is an isomorphism. In fact, by the assumption (i), $f$ is  non-characteristic for $\mS(F)$. By \cite[Proposition 5.3.8]{Lu2}, $\omega_{Y/X} \otimes f^{-1}F \to f^{!}F$ is an isomorphism.
Applying $\mu_{\chi^N}$ we have an isomorphism $\mu_{\chi^N}(\omega_{Y/X} \otimes f^{-1}F) \to \mu_{\chi^N}(f^{!}F)$.

Thus it is enough to show the first morphism in the theorem is an isomorphism. In particular, it is enough to show
\[
R{}^t \kern -2pt f_{\chi!}^{\prime}(\omega_{N/M}\otimes f^{-1}_{\chi \pi}\mu_{\chi_X}(F))|_{\pi^{-1}(V)} \to \mu_{\chi_Y}(\omega_{Y/X} \otimes f^{-1}F)|_{\pi^{-1}(V)}\simeq \mu_{\chi_Y}^{}(f^{!}F)|_{\pi^{-1}(V)}
\]
is an isomorphism. We have
\begin{eqnarray}
T_\chi f^{-1}\nu^{sa}_{\chi_X} (F) &\simeq& T_\chi f^{-1} s^{-1}_X R\Gamma_{\Omega_X} p^{-1}_XF\\
&\simeq& s^{-1}_Y \tilde{f}^{\prime-1} R\Gamma_{\Omega_X} p^{-1}_XF
\end{eqnarray}
and
\begin{eqnarray}
\nu_{\chi_Y}^{sa}(f^{!}F)&\simeq& s^{-1}_YR\Gamma_{\Omega_Y} p^{-1}_Y  f^{!}F\\
&\simeq& s^{-1}_Y \tilde{f}^{\prime!}R\Gamma_{\Omega_Y}p^{-1}_X  F.
\end{eqnarray}
Since there is a natural morphism $\omega_{\tilde{Y}/\tilde{X}} \otimes \tilde{f}^{\prime-1}G \to \tilde{f}^{\prime!} G$ for $G$ we obtain the following distinguished triangle:
\[\begin{tikzcd}
	{\omega_{\tilde{Y}_N/ \tilde{X}_M} \otimes \rho^{-1}\tilde{f}^{\prime-1}R\Gamma_{\Omega_X}p^{-1}_X F} & \rho^{-1}{\tilde{f}^{\prime!}R\Gamma_{\Omega_X}p^{-1}_XF} & H & {{}.}
	\arrow[from=1-1, to=1-2]
	\arrow[from=1-2, to=1-3]
	\arrow["{+1}", from=1-3, to=1-4]
\end{tikzcd}\]
Then $H$ is supported in $S_\chi$. 
And it is enough to prove $\pi^{-1}(V) \cap \mS((s_Y^{-1}H)^\wedge) = \emptyset$.  Thus by the identification \cite[Theorem 3.3]{HPY},  it is enough to prove $\pi^{-1}(V) \cap \mS(H)= \emptyset$. By Theorem \ref{cor6.4.4} it is enough to prove:
\begin{claim}
For any $q_0 \in V$ which has $q \in S_\chi \times_{\tilde{Y}} T^\ast \tilde{Y}$ such that $q_0 = {}^t s_Y^\prime(q) \in V$, we have $\tilde{f}^\prime$ is non-characteristic for $\mS(R\Gamma_{\Omega_X}p^{-1}_XF)$ at $q$.
\end{claim}
Using Lemma \ref{gamesti}, the rest of the proof goes along the same lines in Theorem \ref{multi-inverse} replacing $\mS(F)$ with $W$. In fact, if we assume $\tilde{f}^\prime$ is characteristic for $\mS(R\Gamma_{\Omega_X}p^{-1}_XF)$ at $q$, we get a contradiction.
\endproof
\begin{lem}\label{gamesti}
{\itshape
Let $\Omega_X$ be an open subset in $X$. $F \in D^b(k_{X_{sa}})$ be strongly regular along $V$. Then:
\[
\mS(R\Gamma_{\Omega_X} p^{-1}_X F) \subset {}^t\tilde{p}'_X \tilde{p}^{-1}_{X\pi} V \hat{+} N^\ast (\Omega_X).
\]
In particular,
\[
\mS(\rho^{-1}R\Gamma_{\Omega_X} p^{-1}_X F) \subset {}^t\tilde{p}'_X \tilde{p}^{-1}_{X\pi} V \hat{+} N^\ast (\Omega_X).
\]
}
\end{lem}
\proof
By the strong regularity we have $J(F) \simeq \indlim{i} J(F_i)$ and $\mS(F_i) \subset V$. By \cite[Theorem 3.3]{KS01mic}, $J(R\Gamma_{\Omega_X} p^{-1}_X F) \simeq \indlim{i} J(R\Gamma_{\Omega_X} p^{-1}_X F_i)$ . By \cite[(4.1)]{KS01mic} we have
\begin{eqnarray}
\mS(R\Gamma_{\Omega_X} p^{-1}_X F) &\subset& \lim_i \mS(R\Gamma_{\Omega_X} p^{-1}_X F_i)\\
&=& \lim_i \mS(Rj_{X\ast} \tilde{p}_X^{-1} F_i)\\
&\subset& \lim_i N^\ast(\Omega_X) \hat{+} \mS(\tilde{p}_X^{-1} F_i)\label{Thm6.3.1}\\
&\subset&N^\ast(\Omega_X) \hat{+} {}^t\tilde{p}'_X \tilde{p}^{-1}_{X\pi} V.
\end{eqnarray}
Here \eqref{Thm6.3.1} follows from \cite[Theorem 6.3.1]{KS90}.
\endproof

\section{Applications to $\mathcal{D}$-modules}\label{Sec: D}
\subsection{Solution of $\mathcal{D}$-modules in complex domains with growth conditions.}
Recall the following triangle in \cite[Theorem 1.8]{S}:
\begin{theo}\label{uchida}
{\itshape
Let $F \in D^b(k_{X_{sa}})$ be an object of the derived category of subanalytic sheaves. We have the following distinguished triangle:
\[
	{\nu_{\chi}^{sa}(F)} \to {\tau^{-1}R\Gamma_M(F)\otimes \ori_{M/X}[n]} \to {Rp_{1\ast}R\Gamma_{P^+}p^{-1}_{2}\mu_\chi^{sa} (F)\otimes \ori_{M/X}[n]}\xrightarrow{+1}.
\]
Here $n= \codim M$.
}
\end{theo}
\begin{cor}\label{iso1}
{\it
Let $F \in D^b(k_{X_{sa}} )$. Assume $F$ is strongly regular along $V$. If $\dot{S}_{\chi}^\ast \cap C_{\chi^\ast}(V) =\emptyset$, then
\[
\nu_\chi^{sa}(F) \overset{\sim}{\rightarrow} \tau^{-1}R\Gamma_M(F) \otimes  \ori_{M/X}[n].
\]
}
\end{cor}
\proof
By Theorem \ref{uchida} and Theorem \ref{estimate1}, the result follows.
\endproof
Let $\chi:= \{Y_1, Y_2, \dots, Y_\ell\}$ satisfy the conditions H1, H2 and H3 of [7] and assume each $Y_i$ and $Y:=\bigcap Y_i$ are complex submanifolds of $X$. 
Let $f: Y \hookrightarrow X$ be the canonical embedding. Let $\mathcal{M}$ be a coherent $\mathcal{D}_X$-module.
We define the inverse image of $\mathcal{M}$ by
\[
Df^\ast \mathcal{M}:= \mathscr{O}_Y \underset{f^{-1}\mathscr{O}_X}{\overset{L}{\otimes}} f^{-1}\mathcal{M}.
\]
Assume that $Y$ is non-characteristic for $\mathcal{M}$; that is,
$T_Y^*X \cap {Char} \mathcal{M} \subset T_X^*X$.
Then, it is known that ${D}f^*\mathcal{M}$ is identified with 
$Df^*\mathcal{M} := H^0{D}f^*\mathcal{M}$,
and $Df^*\mathcal{M}$ is a coherent $\mathcal{D}_Y$-module.

\begin{theo}\label{appD: init}
{\it
Assume that $Y$ is non-characteristic for an involutive subbundle $V$ in $T^\ast X$ and $\mathcal{M}$ is regular coherent $\mathcal{D}_X$-module along $V$. Then 
\begin{align*}
\RintHom_{\rho_!\mathcal{D}_X}(\rho_!\mathcal{M}, \nu_\chi^{sa}(\mathscr{O}_X^\lambda)) &\simeq \tau^{-1}\RintHom_{\rho_! \mathcal{D}_Y}(\rho_! Df^\ast \mathcal{M}, \mathscr{O}^\lambda_Y)\\
&\simeq \tau^{-1} f^{-1}\RintHom_{\rho_! \mathcal{D}_X}(\rho_! \mathcal{M}, \mathscr{O}^\lambda_X),
\end{align*}
where $\lambda = w$ or $t$.
}
\end{theo}
\proof
By the assumption, $Y$ is non-characteristic for $\mathcal{M}$. Recall $\mS(\Sol^\lambda(\mathcal{M})) = Char(\mathcal{M})$  by \cite[Corollary 2.2.3]{LucaCKK}.
By \cite[Proposition 5.3.8]{Lu2}
\begin{eqnarray}
\tau^{-1}R\Gamma_Y\RintHom_{\rho_!\mathcal{D}_X}(\rho_!\mathcal{M}, \mathscr{O}_X^\lambda)\otimes \omega^{\otimes-1}_{Y/X} \simeq \tau^{-1} f^{-1}\RintHom_{\rho_!\mathcal{D}_X}(\rho_! \mathcal{M}, \mathscr{O}^{\lambda}_X).
\end{eqnarray}
By \cite[Theorem 3.1.1]{LucaCKK} we have:
\begin{eqnarray}
\tau^{-1} f^{-1}\RintHom_{\rho_!\mathcal{D}_X}(\rho_!\mathcal{M}, \mathscr{O}_X^\lambda) \simeq \tau^{-1}\RintHom_{\rho_!\mathcal{D}_Y}(\rho_!Df^\ast \mathcal{M}, \mathscr{O}^\lambda_Y).
\end{eqnarray}
Since $\text{Sol}^\lambda(\mathcal{M})$ is strongly regular along $V$ by Theorem \ref{reg_sol}, it is enough to check $\dot{S}_{\chi}^\ast\cap C_{\chi^\ast}(V) =\emptyset$ by Corollary \ref{iso1}. By Lemma \ref{non-char}, we have
\[
\dot{S}^*_{\chi} \cap C_{\chi *}(V) = \emptyset.
\]
Hence we obtain the desired result.
\endproof
\begin{cor}\label{app: init1}
{\it
Assume that $Y$ is non-characteristic for an involutive subbundle $V$ in $T^\ast X$ and $\mathcal{M}$ is regular coherent $\mathcal{D}_X$-module along $V$. Then 
\begin{align}
\RintHom_{\mathcal{D}_X}(\mathcal{M}, \nu_\chi(\mathscr{O}_X^\lambda)) &\simeq \tau^{-1}\RintHom_{\mathcal{D}_Y}(Df^\ast \mathcal{M}, \mathscr{O}_Y)\\
&\simeq \tau^{-1} f^{-1}\RintHom_{\mathcal{D}_X}(\mathcal{M}, \mathscr{O}_X),
\end{align}
where $\lambda = w$ or $t$.
}
\end{cor}
\begin{rem}
If $\chi=\{M\}$ is a single submanifold then we do not need $\mathcal{D}$-modules $\mathcal{M}$ to be regular to get the assertion of the Corollary thanks to \cite[Proposition 5.3.4.]{Lu2}. If we take $\lambda = w$ and $\chi$ as a normal crossing divisor, Corollary \ref{app: init1} asserts that the initial value problem for Majima's strongly asymptotically developable functions is uniquely solvable for a non-characteristic regular $\mathcal{D}$-module $\mathcal{M}$.
\end{rem}
\subsection{The division theorems with growth conditions for regular $\mathcal{D}$-modules}
We state division theorems with coefficients of temperate and Whitney multi-microfunctions for regular $\mathcal{D}$-modules.
Let $f: Y \hookrightarrow X$ be a closed embedding of complex manifolds. Let $\chi^N=\{N_1,\dots, N_\ell\}$ be a family of closed submanifolds of $Y$. Let $\chi^M=\{M_1,\dots, M_\ell\}$ be a family of closed submanifolds of $X$.  We also assume the conditions below:
\begin{enumerate}
\item $f(N_j) \subset M_j, j=1, \dots, \ell$,
\item $M_j \subseteq M_{j^\prime}$ iff $N_j \subseteq N_{j^\prime}$.
\end{enumerate}

Assume that these families are of normal crossing type.
Recall the following diagram:
\[\begin{tikzcd}
	{S^\ast_{\chi^N}} & {N \times_M S_{\chi^M}^\ast}& {S_{\chi^M}^\ast} \\
	{T^\ast_N Y} & {N \times_M T^\ast_M X} & {T^\ast_M X,}
	\arrow["\pi","\sim"',from=1-1, to=2-1]
	\arrow["\pi","\sim"',from=1-3, to=2-3]
	\arrow["\sim"',from=1-2, to=2-2]
	\arrow["f_{\chi \pi}",from=1-2, to=1-3]
	\arrow["{{}^t f^\prime_\chi}"', from=1-2, to=1-1]
	\arrow["{{}^t f^\prime_N}"', from=2-2, to=2-1]
	\arrow["{f_{N\pi}}", from=2-2, to=2-3]
\end{tikzcd}\]
where $\pi$ is defined in \eqref{pi}.
\begin{theo}\label{division}
{\itshape
Assume that these families are of normal crossing type.
Let $\mathcal{M}$ be a coherent $\mathcal{D}_X$-module regular along an involutive subbundle $W$. Let $V$ be an open subset of $T^\ast_N Y$.
Assume:
\begin{enumerate}[label=(\roman*)]
\item $f$ is non-characteristic for $W$,
\item $f_{N\pi}: {}^t f_N^{\prime-1}(V) \to T^\ast_M X$ is non-characteristic for $C_{T^\ast_MX}(W)$,
\item $f_{\chi \pi}: {}^t f^{\prime-1}_\chi(\pi^{-1}(V)) \to S_{\chi^{M}}^\ast$ is non-characteristic for $C_{\chi^{M\ast}}(W)$,
\item ${}^t f^{\prime-1}(V)\cap f^{-1}_\pi (W) \subset Y \times_X T^\ast_M X$,
\item ${}^t  f'_{\chi}$ is proper on ${}^t f^{\prime-1}_\chi(\pi^{-1}(V))\cap f^{-1}_{\chi \pi}(C_{\chi^{M\ast}}(W)\cap S^\ast_{\chi^M})$.
\end{enumerate}

Then we have:
\[
\mathrm{R}{}^tf^\prime_{\chi\ast}(f^{-1}_{\chi\pi}\RintHom_{\mathcal{D}_X}(\mathcal{M}, \mu_{\chi^M}(\mathscr{O}^\lambda_X))\otimes \ori_{S_{\chi^{M\ast}}})[d] \xrightarrow{\sim} \RintHom_{\mathcal{D}_Y}(Df^\ast\mathcal{M}, \mu_{\chi^N}(\mathscr{O}^\lambda_Y)\otimes \ori_{S_{\chi^{N\ast}}})
\]
on $\pi^{-1}(V)$. Here $d:= d^{\mathbb{C}}_X-d^{\mathbb{C}}_Y-(d_M-d_N)$, $\lambda= \emptyset, t$ or $w$.
}
\end{theo}
\proof
Let $F:= \RintHom_{\rho_! \mathcal{D}_X}(\rho_!\mathcal{M}, \mathscr{O}^\lambda_X)$. By Theorem \ref{reg_sol}, $F$ is strongly regular along $W$. By Theorem \ref{multi-inverse-sub} we have
\[
R{}^t f^{\prime}_{\chi \ast}(\omega_{N/M}\otimes f_\pi^{-1}\mu_{\chi^M}(F)) \xrightarrow{\sim} \mu_{\chi^N}(\omega_{Y/X} \otimes f^{-1}F)
\]
on $\pi^{-1}(V)$.
Finally by the assumption (i) and Cauchy-Kowaleskaya-Kashiwara theorem with growth conditions \cite[Theorem 3.1.1]{LucaCKK} we obtain
\begin{eqnarray}
\mu_{\chi^N} (f^{-1}F) &\simeq& \mu_{\chi^N}( \RintHom_{\rho_! \mathcal{D}_Y}(\rho_!Df^\ast \mathcal{M}, \mathscr{O}^\lambda_Y)) \\
&\simeq& \RintHom_{\mathcal{D}_Y}(Df^\ast \mathcal{M}, \mu_{\chi^N}(\mathscr{O}^\lambda_Y)).
\end{eqnarray}
\endproof
\begin{rem}
If we suppose $\lambda = \emptyset$, $\chi$ is a single submanifold and $\mathcal{M}$ is elliptic, one recovers \cite[Theorem 1]{KK}. If we suppose $\lambda = t$, $\chi$ is a single submanifold and $\mathcal{M}$ is elliptic, one recovers \cite[Theorem 1]{AndTose}.
\end{rem}
\begin{defi}
We say $\chi$ is a simultaneously linearizable family compatible with complexification if there exists a local coordinate transformation such that we can have the situation described in \cite[Section 4.3]{HPYaxv}.
\end{defi}
\begin{cor}\label{vanis}
{\itshape
Let $\chi^N$ be a simultaneously linearizable family compatible with complexification.
Assume $\mathcal{M}$ is a coherent $\mathcal{D}_X$-module regular along an involutive subbundle $V$ and $M=N$. Let $W$ be a open set in $T^\ast_N Y$. Let $Y$ be non-characteristic for $V$ and assume: 
\begin{enumerate}[label=(\roman*)]
\item ${}^t f^{\prime-1}(W)\cap f^{-1}_\pi (V) \subset Y \times_X T^\ast_M X$.
\end{enumerate}
For $j < \codim_{\mathbb{C}} M$, we have
\[
 R^j\intHom_{\mathcal{D}_X}(\mathcal{M}, \mu_{\chi^M}(\mathscr{O}^\lambda_X) )\simeq 0
\] 
on $\pi^{-1}(W)$. Here $\codim_{\mathbb{C}} M$ is the complex codimension of the maximal complex linear subspace contained in $M$ under the local coordinate system described in Definition 2.4.8.
}
\end{cor}
\proof
Since $N =M$, $f_{N\pi}: N \times_M T^\ast_M X \to T^\ast_M X$ and $f_{\chi \pi}: N \times_M S_{\chi^M}^\ast \to S_{\chi^{M}}^\ast$ are isomorphisms.

Note that $N \times_M T^\ast_M X \subset Y \times_X T^\ast_Y X$.
By the assumption (i), $f_{\pi}^{-1}(W) \cap N \times_M T^\ast_M X \subset f_{\pi}^{-1}(W) \cap Y \times_X T^\ast_Y X \subset T^\ast_X X$. By Lemma \ref{non-char}, we have $C_{\chi^{M\ast}}(W) \cap \dot{S}^\ast_{\chi^{M}} = \emptyset$.
Since $\supp (\mu_{\chi^M}(F)) \subset C_{\chi^{M\ast}}(W) \cap {S}^\ast_{\chi^M}$, ${}^t \kern -2pt f^\prime_{\chi}$ is proper on ${}^t f^{\prime-1}_\chi(\pi^{-1}(V))\cap \supp(f^{-1}_{\chi\pi} \mu^{}_{\chi^M}F )$. 

By Theorem \ref{division} we have,
\[
\mathrm{R}{}^tf^\prime_{\ast} \RintHom_{\mathcal{D}_X}(\mathcal{M}, \mu_{\chi^M}(\mathscr{O}_X^\lambda))[d] \simeq \RintHom_{\mathcal{D}_Y}(Df^\ast \mathcal{M}, \mu_{\chi^N}(\mathscr{O}_Y^\lambda)),
\]
on $\pi^{-1}(W)$.
It is well-known that $H^k (Df^\ast \mathcal{M}) =0$ for $k\not=0$ and $Df^\ast \mathcal{M}$ is a coherent $\mathcal{D}_Y$-module.
From \cite[Section 4]{HPYaxv}, we have $H^k(\mu_{\chi^N} (\mathscr{O}_Y^\lambda)) = 0$ for $k <\codim_{\mathbb{C}} N$.
Thus we obtain that $\mathrm{R}^j\intHom(Df^\ast \mathcal{M}, \mu_{\chi^N} (\mathscr{O}_Y^\lambda)) \simeq 0$ for $j<\codim_{\mathbb{C}} N$. Therefore for any $p\in W$ and for $j< d+ \codim_{\mathbb{C}} N$
\[
(\mathrm{R}^j{}^tf^\prime_{\chi\ast} \RintHom_{\mathcal{D}_X}(\mathcal{M}, \mu_{\chi^M}(\mathscr{O}_X^\lambda)))_{{}^tf^\prime_{\chi}(p)} \simeq 0.
\]
By the finitenes of ${}^tf^\prime_{\chi}$ we obtain
\[
 \mathrm{R}^j\intHom_{\mathcal{D}_X}(\mathcal{M},  \mu_{\chi^M}(\mathscr{O}_X^\lambda)) \simeq 0,
\]
for $j< d+ \codim_{\mathbb{C}} N$.
\endproof
\begin{cor}[Bochner's tube theorem]
{\itshape Let $\chi^M$ be a simultaneously linearizable family compatible with complexification.
Assume $\mathcal{M}$ is a coherent $\mathcal{D}_X$-module regular along an involutive subbundle $V$. Let $Y$ be non-characteristic for $V$.
Let $U$ be an open multi-conic subset of $S_{\chi^M}$. Let $\tilde{U}$ be a convex hull of $U$ in each fiber. We set $W:= S_{\chi^M}^\ast \setminus U^{\circ a}$ and assume: 
\begin{enumerate}[label=(\roman*)]
\item ${}^t f^{\prime-1}(\pi(W))\cap f^{-1}_\pi (V) \subset Y \times_X T^\ast_M X$.
\end{enumerate}
We have:
\[
\Gamma(\tilde{U}; \mathrm{R}^0\intHom_{\mathcal{D}_X}(\mathcal{M}, \nu_{\chi^M} (\mathscr{O}_X^\lambda)))\simeq \Gamma(U; \mathrm{R}^0\intHom_{\mathcal{D}_X}(\mathcal{M}; \nu_{\chi^M}(\mathscr{O}_X^\lambda))).
\]
}
\end{cor}
\proof
By the following triangle obtained from \cite[Theorem 1.8]{S} we get the result. 
\[\begin{tikzcd}
	{\RintHom_{\mathcal{D}_X}( \mathcal{M}, \nu_{\chi}(\mathscr{O}^\lambda_X))} & {\tau^{-1}\RintHom_{\mathcal{D}_X}(\mathcal{M},\rho^{-1}\mathrm{R}\Gamma_M(\mathscr{O}^\lambda_X))\otimes \omega^{\otimes-1}_{M/X}} && {{}} \\
	{{}} & {Rp^{+}_{1\ast}(p^{+}_2)^{-1} \RintHom_{\mathcal{D}_X}(\mathcal{M}, \mu_\chi(\mathscr{O}^\lambda_X))\otimes \omega^{\otimes-1}_{M/X}} & {{}}
	\arrow[from=1-1, to=1-2]
	\arrow[shorten <=46pt, from=2-1, to=2-2]
	\arrow["{+1}", from=2-2, to=2-3]
\end{tikzcd}\label{triangle}
\]
In fact,  by Corollary \ref{vanis} we obtain the following exact sequence:
\[
\begin{tikzcd}
	0 \to \Gamma(U; {\mathrm{R}^0\intHom_{\mathcal{D}_X}( \mathcal{M},\nu_{\chi}(\mathscr{O}^\lambda_X)))} & {\Gamma(M; R^d \intHom_{\mathcal{D}_X}(\mathcal{M},\rho^{-1}\mathrm{R}\Gamma_M(\mathscr{O}^\lambda_X))\otimes \ori_{M/X})} && {{}} \\
	{{}} & {\Gamma(S_\chi^\ast\setminus U^{\circ a}; \mathrm{R}^d\intHom_{\mathcal{D}_X}(\mathcal{M}, \mu_\chi(\mathscr{O}^\lambda_X))\otimes \ori_{M/X}).}
	\arrow[from=1-1, to=1-2]
	\arrow[shorten <=46pt, from=2-1, to=2-2]
\end{tikzcd}
\]
Similarly, we have an exact sequence as above with $U$ replaced with $\tilde{U}$.
Since $U^{\circ a}=\tilde{U}^{\circ a}$, by the five lemma the assertion follows.
\endproof
\begin{rem}
This is a generalization of the work \cite{U, AT, SKK}.
\end{rem}
\section*{Acknowledgements}
We thank Prof. N. Honda, Prof. L. Prelli and Prof. S. Yamazaki for generous supports and helpful comments through the preparation of this work.

\end{document}